\documentclass[12pt,letterpaper]{article}
\usepackage{amsmath,amsfonts,amsthm,amssymb}

\swapnumbers 

\newtheorem{lemma}{Lemma}[section]

\newtheorem{theorem}[lemma]{Theorem}

\theoremstyle{definition} 

\newcommand{\Nat}{{\mathbb N}}

\newcommand\reals{{\mathbb R}}

\begin{document}

\title{Handelman's theorem for an order unit normed space}

\author{\textsc{David J. Foulis}\\[2mm]
{\small Department of Mathematics and Statistics,
University of Massachusetts, }\\
{\small 1 Sutton Court, MA 01002, Amherst, USA}\\[2.5mm]
\textsc{Sylvia Pulmannov\'{a}}\\[2mm]
{\small Mathematical Institute, Slovak Academy of Sciences,}\\
{\small \v Stef\'anikova 49, SK-814 73 Bratislava, Slovakia}\\
}

\date{ }

\maketitle

\begin{abstract}
We give a detailed proof D. Handelman's theorem stating (in the
context of an order unit normed space) that a monotone $\sigma$-complete
order unit normed space is a Banach space.
\end{abstract}

\section{Handelman's theorem}

In what follows, \emph{$(V,u)$ denotes an order unit normed space}. By
definition, $V$ is \emph{monotone $\sigma$-complete} iff every
ascending sequence in $V$ that is bounded above has a supremum in $V$.
Our proof of the following theorem is based on the proof of \cite
[Proposition 3.9]{Hand}.

\begin{theorem} {{\rm[Handelman]}}
If $(V,u)$ is monotone $\sigma$-complete, than $V$ is norm complete
{\rm(}i.e., a Banach space{\rm)}.
\end{theorem}

\begin{proof}
Suppose that $V$ is monotone $\sigma$-complete and let $(a\sb{n})
\sb{n=0}\sp{\infty}$ be a Cauchy sequence in $V$. Then there
exists $0<\beta\in\reals$ such that $-\beta u\leq a\sb{n}\leq\beta u$
for all $n=0,1,2,...$, and by replacing each $a\sb{n}$ by $\beta\sp{-1}
a\sb{n}$, we can assume that $-u\leq a\sb{n}\leq u$, i.e., $\|a\sb{n}\|
\leq 1$, for $n=0,1,2,...$. Also, we can and do replace $a\sb{0}$ by $0$
without affecting the hypothesis that $(a\sb{n})\sb{n=0}\sp{\infty}$ is a
Cauchy sequence.

For each $n\in\Nat$, there exists $M\sb{n}\in\Nat$ such that
\[
M\sb{n}\leq i,j\in\Nat\Rightarrow \|a\sb{i}-a\sb{j}\|<2\sp{-n},
\]
and we can assume without loss of generality that
\[
M\sb{1}<M\sb{2}<M\sb{3}<\cdots.
\]
Thus, for $n=1,2,3,...$,
\[
\|a\sb{M\sb{n+1}}-a\sb{M\sb{n}}\|<2\sp{-n}.
\]
It will be sufficient to prove that the subsequence of $(a\sb{n})
\sb{n=0}\sp{\infty}$ given by
\[
a\sb{0},\, a\sb{M\sb{1}},\, a\sb{M\sb{2}},\, a\sb{M\sb{3}}, ...
\]
converges. Replacing $(a\sb{n})\sb{n=0}\sp{\infty}$ by this
subsequence, we have
\[
a\sb{0}=0\text{\ and for all\ }n\in\Nat,\ \|a\sb{n}\|\leq 1
\text{\ and\ }\|a\sb{n+1}-a\sb{n}\|<2\sp{-n}.
\]
Thus,
\setcounter{equation}{0}
\begin{equation} \label{eq:H01}
a\sb{0}=0,\,-u\leq a\sb{n}\leq u\text{\ and\ }-2\sp{-n}u<a\sb{n+1}-a\sb{n}
 <2\sp{-n}u\text{\ for all\ }n\in\Nat.
\end{equation}

Now put
\[
b\sb{n}:=2\sp{-n}u+a\sb{n+1}-a\sb{n}\text{\ for\ }n=0,1,2,3,... .
\]
In particular, $b\sb{0}=u+a\sb{1}$, and since $-u\leq a\sb{1}
\leq u$, we have $0\leq u+a\sb{1}\leq 2u$, whence $0\leq b\sb{0}
\leq 2u$. Also, for $n\in\Nat$, $0\leq b\sb{n}\leq 2(2\sp{-n})u$
and therefore
\begin{equation} \label{eq:H02}
0\leq b\sb{n}\leq 2(2\sp{-n})u\text{\ for\ }n=0,1,2,3,... .
\end{equation}
Consider the partial sums
\begin{equation} \label{eq:H03}
s\sb{m}:=\sum\sb{n=0}\sp{m}b\sb{n}=a\sb{m}+2u-2\sp{-m}u.
\end{equation}
By (\ref{eq:H02}), $(s\sb{m})\sb{m=0}\sp{\infty}$ is monotone
increasing and
\begin{equation} \label{eq:H04}
0\leq s\sb{m}=\sum\sb{n=0}\sp{m}b\sb{n}\leq\sum\sb{n=0}\sp{m}2
(2\sp{-n})u\leq 4u-2\sp{1-m}u\leq 4u\text{\ for\ }m=0,1,2,3,...\ ;
\end{equation}
therefore $s:=\bigvee\sb{m=0}\sp{\infty}s\sb{m}$ exists in $V$.

\newpage

Temporarily fix $m\in\{0,1,2,3,...\}$.  Then
\begin{equation} \label{eq:H05}
\text{for\ }m<p\in\Nat,\ 0\leq s\sb{p}-s\sb{m}=\sum\sb{k=m+1}
 \sp{p}b\sb{k}.
\end{equation}
Thus $(s\sb{p}-s\sb{m})\sb{p=m+1}\sp{\infty}$ is a monotone
increasing sequence in $V$, and by (\ref{eq:H04}),
\[
0\leq s\sb{p}-s\sb{m}\leq s\sb{p}\leq 4u\text{\ for\ }m<p\in\Nat,
\]
whence $\bigvee\sb{p=m+1}\sp{\infty}(s\sb{p}-s\sb{m})$ exists in $V$
for $m=0,1,2,3,...\ .$

\smallskip

Since $(s\sb{n})\sb{n=0}\sp{\infty}$ is monotone increasing, it follows
that
\begin{equation} \label{eq:H06}
s=\bigvee\sb{n=0}\sp{\infty}s\sb{n}=\bigvee\sb{p=m+1}\sp{\infty}s\sb{p}
\text{\ for\ }m=0,1,2,3,...\ ,
\end{equation}
and by (\ref{eq:H06}) and (\ref{eq:H05}), we have
\begin{equation} \label{eq:H07}
0\leq s-s\sb{m}=\left(\bigvee\sb{p=m+1}\sp{\infty}s\sb{p}\right)-s\sb{m}=
\bigvee\sb{p=m+1}\sp{\infty}(s\sb{p}-s\sb{m})=\bigvee\sb{p=m+1}
 \sp{\infty}\left(\sum\sb{k=m+1}\sp{p}b\sb{k}\right).
\end{equation}
By (\ref{eq:H02}), for all $m=0,1,2,3,...$ and all $p>m$,
\[
\sum\sb{k=m+1}\sp{p}b\sb{k}\leq\sum\sb{k=m+1}\sp{p}2\cdot2\sp{-k}u
\leq\sum\sb{k=m+1}\sp{\infty}2\cdot 2\sp{-k}u=2\cdot2\sp{-m}u,
\]
whence
\[
\bigvee\sb{p=m+1}\sp{\infty}\left(\sum\sb{k=m+1}\sp{p}b\sb{k}\right)
 \leq 2\cdot2\sp{-m}u.
\]
Consequently, by (\ref{eq:H07}), $0\leq s-s\sb{m}\leq 2\cdot 2\sp{-m}u$,
and therefore $\|s-s\sb{m}\|\leq 2\cdot 2\sp{-m}$ for $m=0,1,2,...\ $ and
we have
\begin{equation} \label{eq:H08}
s=\lim\sb{m\rightarrow\infty}s\sb{m}.
\end{equation}
Thus by (\ref{eq:H03}) and (\ref{eq:H08}),
\[
\lim\sb{m\rightarrow\infty}a\sb{m}=\lim\sb{m\rightarrow\infty}(s\sb{m}
+2\sp{-m}u-2u)=s-2u. \qedhere
\]
\end{proof}

\end{document}